\documentclass[letterpaper,11pt,oneside,reqno]{amsart}

\usepackage{xcolor}

\usepackage{bm}
\usepackage{mathrsfs}
\usepackage{yfonts}
\usepackage{amsfonts,amsmath, amssymb,amsthm,amscd}
\usepackage[height=9.6in,width=5.95in]{geometry}
\usepackage[mpexclude,DIV13]{typearea}
\usepackage{verbatim}
\usepackage{hyperref}
\usepackage{graphicx}
\usepackage[latin1]{inputenc}
\usepackage{latexsym}
\usepackage{lscape}
\usepackage{epsfig}
\usepackage{subfigure}
\include{bibtex}

\usepackage{setspace}

\begin{document}
\bibliographystyle{alpha}
\newcommand{\cn}[1]{\overline{#1}}
\newcommand{\e}[0]{\epsilon}

\newcommand{\ZOY}[3]{\ensuremath{\mathcal{Z}^{#1}_{#2}(#3)}}

\newcommand{\whitenoise}{\ensuremath{\mathscr{\dot{W}}}}

\newcommand{\EE}{\ensuremath{\mathbb{E}}}
\newcommand{\PP}{\ensuremath{\mathbb{P}}}
\newcommand{\var}{\textrm{var}}
\newcommand{\N}{\ensuremath{\mathbb{N}}}
\newcommand{\R}{\ensuremath{\mathbb{R}}}
\newcommand{\C}{\ensuremath{\mathbb{C}}}
\newcommand{\Z}{\ensuremath{\mathbb{Z}}}
\newcommand{\Q}{\ensuremath{\mathbb{Q}}}
\newcommand{\T}{\ensuremath{\mathbb{T}}}
\newcommand{\E}[0]{\mathbb{E}}
\newcommand{\OO}[0]{\Omega}
\newcommand{\F}[0]{\mathfrak{F}}
\def \Ai {{\rm Ai}}
\newcommand{\G}[0]{\mathfrak{G}}
\newcommand{\ta}[0]{\theta}
\newcommand{\w}[0]{\omega}
\newcommand{\ra}[0]{\rightarrow}
\newcommand{\vectoro}{\overline}
\newcommand{\crairy}{\mathcal{CA}}
\newcommand{\wxy}{\mathcal{W}_{k;\bar{x},\bar{y}}}
\newtheorem{theorem}{Theorem}[section]
\newtheorem{partialtheorem}{Partial Theorem}[section]
\newtheorem{conj}[theorem]{Conjecture}
\newtheorem{lemma}[theorem]{Lemma}
\newtheorem{proposition}[theorem]{Proposition}
\newtheorem{corollary}[theorem]{Corollary}
\newtheorem{claim}[theorem]{Claim}
\newtheorem{experiment}[theorem]{Experimental Result}

\def\todo#1{\marginpar{\raggedright\footnotesize #1}}
\def\change#1{{\color{green}\todo{change}#1}}
\def\note#1{\textup{\textsf{\color{blue}(#1)}}}

\theoremstyle{definition}
\newtheorem{rem}[theorem]{Remark}

\theoremstyle{definition}
\newtheorem{com}[theorem]{Comment}

\theoremstyle{definition}
\newtheorem{definition}[theorem]{Definition}

\theoremstyle{definition}
\newtheorem{definitions}[theorem]{Definitions}

\theoremstyle{definition}
\newtheorem{conjecture}[theorem]{Conjecture}

\title{Universality for directed polymers in thin rectangles}

\author[A. Auffinger]{Antonio Auffinger}
\address{A. Auffinger,
University of Chicago,
5734 S. University Avenue, Chicago, IL 60637, USA}
\email{auffing@math.uchicago.edu}

\author[J. Baik]{Jinho Baik}
\address{J. Baik,
Department of Mathematics, University of Michigan,
 530 Church Street,
Ann Arbor, MI 48109, USA}
\email{baik@umich.edu}

\author[I. Corwin]{Ivan Corwin}
\address{I. Corwin,
Microsoft Research,
New England, 1 Memorial Drive, Cambridge, MA 02142, USA}
\email{ivan.corwin@gmail.com}

 \maketitle
\begin{abstract}
We consider the fluctuations of the free energy of positive temperature directed polymers
in thin rectangles $(N,N^{\alpha})$, $\alpha <3/14$. For general weight distributions with finite fourth moment we prove that the distribution of these fluctuations converges as $N$ goes to infinity to the GUE Tracy-Widom  distribution.
\end{abstract}

\section{Introduction}


We prove a statement of KPZ universality for directed polymers in random media in thin rectangles $(N,N^{\alpha})$ for $\alpha<3/14$. We assume \emph{general} weight distributions with at least a finite fourth moment. Our main result is the positive temperature analogue of Corollary 1.1 of \cite{JinhoToufic}, Theorem 1 of \cite{BM}, and the result of \cite{Suidan}, where a similar universality for last passage percolation was proved.

%
The model we consider is defined as follows.

\begin{definition}\label{def:1}
Let $W_{ij}$, $i,j \in \mathbb N$, be a family of i.i.d. random variables with $\E[W_{11}]=0$,  $\E[ (W_{11})^2] =1$ and $\E[(W_{11})^4] < \infty$. For each $j,k \in \mathbb N$ we define $S^j(k) = \sum_{i=1}^{k} W_{ij}$ with the convention $S^j(0)=0$. The partition function for the discrete directed random polymer from $(1,1)$ to $(N,n)$ at  inverse temperature $\beta>0$ is defined as
\begin{equation}\label{def:discretePolymer}
	Z_{N,n}(\beta) =
	\sum_{1=i_0\leq i_1 \leq \ldots \leq i_n=N}
	\exp \bigg(\beta \sum_{j=1}^n ( S^j(i_j)-S^j(i_{j-1}-1) ) \bigg).
\end{equation}

\end{definition}
\begin{rem} The hypotheses that $\E[W_{11}]=0$,  $\E[(W_{11})^2] =1$ are not restrictive. Indeed, a non-zero mean only changes $Z_{N,n}(\beta)$ by a (deterministic) multiplicative constant while a different second moment is just a rescaling of $\beta$. Also note that one can rewrite $Z_{N,n}(\beta)$ as the sum over all up/right lattice paths $\pi$ from $(1,1)$ to $(N,n)$ of the Boltzmann weights $\exp(\beta \sum_{(i,j)\in \pi} W_{ij})$.
\end{rem}

The main result of the paper is the following:
\begin{theorem}\label{mainthm}
For all $\alpha \in (0,\frac{3}{14})$ and $\beta>0$, if $n = \lfloor N^{\alpha}\rfloor$ then
\begin{equation*}
	\lim_{N\rightarrow \infty} \PP \bigg( \frac{\log Z_{N,n}(\beta) - 2\beta N^{\frac12+\frac{\alpha}2}}{\beta N^{\frac12-\frac{\alpha}6} } \leq r\bigg)  = F_{GUE}(r),
\end{equation*}
for each fixed $r\in \R$,
where $F_{GUE}$ is the GUE Tracy-Widom distribution function \cite{TW}.
\end{theorem}

\begin{rem}\label{rem:KMT}
If we assume that the $W_{ij}$ have all moments finite, the above result holds for all $\alpha<\frac37$.
See Remark \ref{rem:KMT2} below.
\end{rem}

%
%

The thinness assumption, $n=\lfloor N^{\alpha}\rfloor$ with $\alpha\in (0, \frac{3}{14})$, is only a technical condition. It is~anticipated~that a universality result should also hold for thicker rectangles (for all $0<\alpha\le 1$), after an appropriate modification of the centering and scaling of $\log Z_{N,n}(\beta)$.
The modification should depend on the distribution of weights; however, there is presently no prediction for the exact dependence for the full rectangle case when $n=N$\footnote{There is presently only one known {\it exactly solvable} positive temperature discrete polymer. It has weights which are distributed as the logarithm of inverse Gamma random variables. Introduced by Sepp\"{a}l\"{a}inen \cite{Sep}, this polymer has an explicit product invariant measure which allowed him to compute the law of large numbers for the free energy as well as a tight upper bound on the free energy fluctuation scaling exponent for this model. Due to the tropical RSK correspondence \cite{COSZ}, this polymer also fits into the hierarchy of solvable models studied in \cite{BorCor}. Though it has not yet been done, this should enable a rigorous proof of the GUE Tracy-Widom limit theorem for any size rectangle as long as both $n$ and $N$ go to infinity. Due to the work of \cite{OCon}, another model which fits into the hierarchy of \cite{BorCor} is the semi-discrete directed polymer model, introduced by O'Connell-Yor \cite{OY} (see discussion below Proposition \ref{asyprop}).},
For $\alpha<1$, at least the leading order term of the centering was already known. Based on the earlier works of \cite{GW,Sep2} and estimates similar to some of those in the proof of Proposition~\ref{compareprop}, Moreno Flores \cite[Theorem (8.8)]{Moreno} showed a law of large numbers for the free energy by proving that with probability one:
\begin{equation}\label{eq:LLN}
\lim_{N\rightarrow \infty} \frac{1}{N^{\frac{1}{2}+\frac{\alpha}{2}}}\log Z_{N,n}(\beta) = 2 \beta.
\end{equation}

We prove Theorem \ref{mainthm} by employing the Skorohod embedding theorem to couple the discrete directed polymers to the semi-discrete directed polymer of O'Connell-Yor\cite{OY}.
The asymptotics of the free energy of the semi-discrete polymer is then evaluated via a steepest descent analysis of the Fredholm determinant formula of \cite[Theorem 5.2.10]{BorCor}.

This coupling approach is an adaptation of the previous works on directed last passage percolation by \cite{JinhoToufic} and \cite{BM}, which corresponds to the case when $\beta=\infty$.
We note that the same thinness condition $\alpha<3/14$ was assumed in these works under the finite fourth moment condition. Let $L_{N,n} = \lim_{\beta\to\infty} \beta^{-1} \log Z_{N,n}(\beta)$ denote the last passage percolation time. Then it was shown that
\begin{equation}\label{eq:LPPuniv}
	\lim_{N\rightarrow \infty} \PP \bigg( \frac{L_{N,n} - 2 \sqrt{Nn}}{\sqrt{N}n^{-1/6} } \leq r\bigg)  = F_{GUE}(r).
\end{equation}
This is consistent with our main result if we formally take the limit $\beta\to \infty$ and replace $\beta^{-1} \log Z_{N,n}(\beta)$ by $L_{N,n}$.

When we assume that the first $p$ moments are finite for $p>2$, then it was shown in \cite{BM} that~\eqref{eq:LPPuniv} holds for $\alpha<\frac{6}{7}(\frac12- \frac1{p})$.
This is obtained by using the Koml\'{o}s-Major-Tusn\'{a}dy  theorem instead of the Skorohod embedding theorem in the couple argument.
Note that the upper bound is $\frac37$ when $p=\infty$, which is in agreement with the result in this paper stated in Remark~\ref{rem:KMT} above.
We remark that there is a very different proof of~\eqref{eq:LPPuniv} by \cite{Suidan} assuming that $p=3$. This proof is based on a concentration inequality of \cite{Chatterjee}.
The condition on $\alpha$ is same as $\alpha<\frac{6}{7}(\frac12- \frac1{p})$ with $p=3$.
It is an interesting question to find this approach can also be adapted to the polymer case.


\medskip

The semi-discrete directed random polymer of O'Connell-Yor \cite{OY} is defined as follows.

\begin{definition}
Fix $t>0$ and $n\in \N$.
To each ordered set $0=t_0 < t_1 < \cdots  <t_{n-1} < t_n=t$, we associate an up/right semi-discrete path $\phi=\{(x, i): t_{i-1}\le x\le t_i\} $  in $[0,t] \times \{1, \ldots, n\}$.
Given a family of independent standard one dimensional Brownian motions $B^i$, $1\leq i \leq n$, we define the energy of an up/right path $\phi$ as
$$ E(\phi) = B^1(t_1) + (B^2(t_2)-B^2(t_1)) + \cdots + (B^n(t_n)-B^n(t_{n-1})).$$
The  partition function for the semi-discrete directed random polymer at inverse temperature $\beta>0$  is
then defined as
\begin{equation*}\label{def:semidiscrete}
	\ZOY{n}{t}{\beta} = \int\limits_{0 =t_0< t_1 < \cdots  <t_{n-1} < t_n=t} \exp \big(\beta E(\phi)\big) \; d\phi \end{equation*}
where $d\phi$ denotes the Lebesgue measure in over the variables $t_1,\ldots, t_{n-1}$.
\end{definition}

\bigskip

Theorem \ref{mainthm} is a consequence of the following two results which we prove in Section \ref{sec:proofs}.
The first is a coupling result of $Z_{N,n}(\beta)$ and $\ZOY{n}{N}{\beta}$.
Here we do not assume anything on the relationship between $n$ and~$N$.

%

\begin{proposition}\label{compareprop}
Given $\delta>0$, there exist positive constants $C=C(\delta, \beta)$ and $c=c(\beta)$ such that
\begin{equation*}
	\PP \bigg( \big |\log Z_{N,n}(\beta) - \log \ZOY{n}{N}{\beta}  \big|>a \bigg)
	\leq \frac{C n N^{\frac14+\delta}}{a} + CNne^{-c(a-\log(n!))^2/n^2} 
\end{equation*}
for all $n,N\in \N$ and $a>0$.
\end{proposition}

\begin{rem}\label{rem:KMT2}
If we assume that the $W_{ij}$ have all moments finite, then we can couple the discrete and semi-discrete polymers using the Koml\'{o}s-Major-Tusn\'{a}dy theorem and replace the term
$N^{\frac14+\delta}$ to $N^{\delta}$. For further detail, see Remark \ref{rem:KMTproof}.
It is easy to check that this improvement is enough to prove Theorem \ref{mainthm} for $\alpha<\frac37$.
\end{rem}

The second result is on the asymptotics for $\ZOY{n}{N}{\beta}$.

\begin{proposition}\label{asyprop}
For each fixed $\alpha\in (0,1)$ and $\beta>0$,
\begin{equation*}
	\lim_{t\to \infty} \PP\bigg(\frac{\log \ZOY{\lfloor t^{\alpha}\rfloor}{\beta^2 t}{1} - 2\beta t^{1-\kappa}}{\beta t^{\mu}}\leq r\bigg)= F_{\rm{GUE}}(r)
\end{equation*}
for every fixed $r\in\R$,
where $\kappa = \frac{1-\alpha}{2}$ and $\mu = \frac{3-\alpha}{6}$.
\end{proposition}

Note that for $\alpha\in (0,1)$ we have that $\kappa \in (0,1/2)$ and $\mu\in (1/3,1/2)$.
For $\alpha=1$ an analog to the above result with $\mu=1/3$ is proved in \cite[Theorem 5.2.12]{BorCor} and \cite[Theorem 2.1]{BorCorFer}.
In this case the form of the law of the large numbers term which must be subtracted is
more involved.
Our proof of Proposition~\ref{asyprop} is an adaptation of the analysis of \cite{BorCor,BorCorFer}.
On the other hand,
for $\alpha=0$, if the number of rows is kept fixed at $n$ then the semi-discrete polymer free energy converges (after $\mu=1/2$ scaling) to a semi-discrete last passage percolation time.
It is known that
this last passage time is exactly distributed as the largest eigenvalue of an $n\times n$ GUE random matrix
(see \cite{Bary, GTW}; see also Theorem 1.1 of \cite{OCon}).
Thus, when $\alpha=0$, if we take $t$ to infinity first and then let $n$ tend to infinity,
we also obtain the $F_{\rm{GUE}}$ distribution as the limiting fluctuation.

We can now prove Theorem \ref{mainthm}.

\begin{proof}[Proof of Theorem \ref{mainthm}]
We apply Proposition~\ref{asyprop} with $t=N$ and note that Brownian scaling implies that $\ZOY{n}{t}{\beta}$ is equal in law to $\beta^{-2(n-1)}\ZOY{n}{\beta^2 t}{1}$. This reduces the proof of Theorem \ref{mainthm} to the claim that
$$\frac{\log Z_{N,n}(\beta) - \log \ZOY{n}{N}{\beta} - 2(N^{\alpha}-1)\log\beta}
{N^{\frac12-\frac{\alpha}6}} \;\quad \text {converges to zero in probability.}$$
We may neglect the term $2(N^{\alpha}-1)\log\beta$ due to the fact that,
\begin{equation}\label{alphacomp}
N^{\alpha}\ll N^{\frac12-\frac{\alpha}6}\quad \textrm{for all } \alpha<3/7
\end{equation}
as $N$ goes to infinity.

For any fixed $\epsilon>0$, we apply Proposition~\ref{compareprop} with $a=\epsilon N^{\frac12-\frac{\alpha}6}$
and $n=\lfloor N^{\alpha}\rfloor$ to see that for all $\delta>0$ there exist positive constants $c$ and $C$ such that
$$\PP \bigg( \left|\frac{\log Z_{N,n}(\beta)  - \log \ZOY{n}{N}{\beta}}{N^{\frac12- \frac{\alpha}{6}}}\right| > \e \bigg) \leq \frac{C}{\epsilon} N^{\frac76\alpha-\frac14+\delta}+ CN^{1+\alpha}e^{-\eta(N)}. $$
where $\eta(N) = cN^{-2\alpha} (a-\log(N^{\alpha}!))^2$. By Stirling's approximation, $\log(N^\alpha !) \approx N^{\alpha}\log N^{\alpha} - N^{\alpha}$. Using (\ref{alphacomp}) again, we find that $a\gg \log(N^{\alpha}!)$ for $N$ large, and hence $\eta(N) \geq c' N^{1-\frac{7}{3} \alpha}$ for some other constant $c'$ depending on $c$ and $\epsilon$.

As $N$ goes to infinity, the second term converges to zero if $\alpha<\frac37$ and the first term converges to zero if $\alpha< \frac{3}{14}$ by choosing $\delta$ to be small enough (but fixed in $N$). This completes the proof of Theorem \ref{mainthm}.
\end{proof}

\bigskip


\subsubsection*{Acknowledgments}
We would like to thank Moreno Flores for bring his work \cite{Moreno} to our attention. The work of Jinho Baik was supported in part by NSF grants DMS1068646. The work of Ivan Corwin was funded by Microsoft Research through the Schramm Memorial Fellowship.

\section{Proofs}\label{sec:proofs}

We prove Proposition \ref{compareprop} and Proposition \ref{asyprop}.

\subsection{Proof of Proposition \ref{compareprop}}
We compare the discrete polymer of general weight distribution (satisfying the conditions of Definition \ref{def:1}) with the semi-discrete polymer by coupling them using the Skorohod embedding theorem.

For much of the proof it will be convenient to rescale our processes so as to be functions in $C([0,1], \mathbb R^n)$, the space of continuous functions from $[0,1]$ to $\mathbb R^n$. We will denote such rescaled processes with an overbar to distinguish them from the unscaled versions. For two such functions $f=(f^1,\ldots,f^n)$ and  $g=(g^1, \ldots, g^n)$ in $C([0,1],\mathbb R^n)$ define the metric
\begin{equation}\label{metric}
d(f,g)=\sum_{j=1}^{n} ||f^j - g^j||_{\infty},
\end{equation}
where
\begin{equation}
||f^j - g^j||_{\infty} = \sup_{t \in [0,1]} |f^j(t) - g^j(t)|.
\end{equation}
Set
\begin{equation}\label{def:F}
	F_N(f)= \log \bigg[  \sum_{\substack{\frac{1}{N}=t_0 \leq t_1 \leq \cdots \leq t_{n-1}\leq t_n = 1 \\ t_i \in \frac1{N}\Z}}
	\exp \bigg( \beta \sum_{j=1}^{n} (f^j(t_j)-f^j(t_{j-1}-\tfrac{1}{N}))\bigg) \bigg]
\end{equation}
for  $N\in \mathbb N$. A basic property of $F_N$ is the following:

\begin{lemma}\label{lem:Lip}
The function $F_N:C([0,1],\mathbb R^n) \rightarrow \R$ defined in \eqref{def:F} is Lipchitz continuous with Lipchitz constant less than or equal to $2\beta$ with respect to the metric~\eqref{metric}.
\end{lemma}

\begin{proof}
Set $A(f)= e^{F_N(f)}$.
Triangle inequality implies that $A(f) \leq A(g)e^{2\beta d(f,g)}$. 
Therefore, assuming without loss of generality that $\frac{A(f)}{A(g)} \geq 1$, we have
$|F(f)-F(g)| = \left|\log \frac{A(f)}{A(g)}\right|  \leq 2\beta d(f,g)$.
\end{proof}

The next lemma gives us a coupling between the random walks $S^j(k)$ in Definition~\ref{def:1} and the Brownian motions $B^j(t)$.
This is the Skorohod embedding theorem (whose proof can essentially be found in \cite[Theorem 8.6.1]{Durrett}) applied to each walk $S^{j}(\cdot)$.

\begin{lemma}\label{Skorohod}
For each $j$, there exist i.i.d. stopping times $\tau^{j}_k$, $k=1,2,\ldots$, (which are measurable with respect to the Brownian filtration of $B^j$) such that $\E[\tau^{j}_k] = 1$, $\E[(\tau^{j}_k)^2]~<~\infty$,  and
\begin{equation*}
	B^j(\tau_1^j+\cdots + \tau_k^j) \stackrel{d}{=}  S^j(k), \qquad k=1,2,\ldots.
\end{equation*}
\end{lemma}

Define
\begin{equation}\label{def:interpolation}
         \bar S_N^j(t) = B^j(\tau_1^j+\cdots + \tau_{\lfloor Nt \rfloor}^j)
         + \big(Nt -\lfloor Nt \rfloor \big)
         \big(B^j(\tau_1^j+\cdots + \tau_{\lceil Nt \rceil}^j) - B^j(\tau_1^j+\cdots + \tau_{\lfloor Nt \rfloor}^j)\big)
\end{equation}
to be the piecewise linear interpolation of the Brownian motion sampled at times $\tau^1_1+\cdots+\tau^j_k$ for $1\leq k \leq N$, rescaled so as to be a function in $C([0,1],\mathbb R)$.

Consider the continuous vector valued function $\bar S_N(t) := (\bar{S}_N^1(t), \ldots, \bar{S}_N^n(t))\in C([0,1],\mathbb R^n)$.
From Lemma~\ref{Skorohod}, $\bar S_N^j(\frac{k}{N}) = S_N^j(k)$ for $1~\le~k~\le~N$. Therefore, from \eqref{def:discretePolymer}, \eqref{def:F}, and \eqref{def:interpolation}, one finds that
\begin{equation*}\label{eq:distributioneq}
	\log Z_{N,n}(\beta) \stackrel{d}{=} F_N\big(\bar S_N\big).
\end{equation*}

We also define the rescaling of the Brownian motions by
$$\bar B_N(t) = (B^1_N(Nt), \ldots, B_N^n(Nt)).$$ 
The proof of Proposition~\ref{compareprop} is now obtained by estimating probabilitistic bounds on
$F_N\big(\bar S_N\big)-F_N\big(\bar B_N\big)$ and
$F_N\big(\bar B_N\big)-\log \ZOY{n}{N}{\beta} $.
The relevant estimates are obtained in Lemma~\ref{firstmainlemma} and in Lemma~\ref{lem:estimate} respectively.

\begin{lemma}\label{firstmainlemma}
For any given $\delta>0$,  there exist a positive constant $C$ such that
\begin{equation*}\label{eq:bound}
	\PP\bigg(\left|F_N(\bar{S}_N)-F_N(\bar{B}_N)\right| \geq a \bigg)  \leq \frac{C\beta n}{a}N^{\frac14+\delta}
\end{equation*}
for all $n, N\in \mathbb N$ and  $a > 0$.
\end{lemma}


\begin{proof}
From Lemma \ref{lem:Lip}, the Markov inequality, and the union bound
\begin{equation}\label{eq:eq9}
\begin{split}
 	\PP\bigg(\left|F_N(\bar{S}_N)-F_N(\bar{B}_N)\right| \geq a \bigg) &\leq \PP\bigg(d(\bar S_N, \bar B_N)  \geq \frac{a }{2 \beta  } \bigg)
\leq \frac{2\beta n}{a} \E \| \bar{S}_N^1-\bar{B}^1_N\|_{\infty}.
\end{split}
\end{equation}

Define the event
\begin{equation*}
	A_N(u,\lambda) := \bigg\{ \sup_{1\leq i \leq N} \left|\sum_{\ell=1}^i \tau_\ell^1 -i\right|>u N^{1-\lambda}\bigg\}
\end{equation*}
for $u>0$ and $\lambda\in (0,1)$. 
From Doob's inequality (\cite[Theorem 5.4.2 and Example 5.4.1]{Durrett}) and the definition of $\tau_l^1$ we have
\begin{equation}\label{eq:doob}
\PP \big( A_N(u,\lambda) \big) \leq \frac{\text{Var} \; (\tau^1_1)}{N^{1-2\lambda}u^2}.
\end{equation}
We now fix $\rho>0$ and write
\begin{equation}\label{eq:eq1}
\begin{split}
 	\E \| \bar{S}_N^1-\bar{B}^1_N\|_{\infty}
	&= N^\rho \int_{0}^{\infty} \PP \bigg(\| \bar{S}_N^1-\bar{B}^1_N\|_{\infty} > uN^{\rho}\bigg) \; du \\
 &\le N^{\rho} + N^\rho \int_{1}^{\infty} \PP \bigg(\| \bar{S}_N^1-\bar{B}^1_N\|_{\infty} > u N^{\rho} \bigg) \; du.
\end{split}
\end{equation}
In order to estimate the last integral, we fix $\lambda\in (0,1)$ and
divide the sample space into $A_N(u,\lambda)$ and $A_N(u,\lambda)^c$.
From  \eqref{eq:doob} we find that
\begin{equation}\label{eq:eq2}
	N^\rho \int_{1}^{\infty}
	\PP \bigg(\|  \bar{S}_N^1-\bar{B}^1_N \|_{\infty} > uN^{\rho}, \; A_N(u,\lambda)\bigg) \; du
	\le \text{Var} \; (\tau^1_1) N^{\rho- 1+2\lambda} .
\end{equation}

It remains to estimate
\begin{equation*}\label{eq:eq1-1-1}
\begin{split}
	N^\rho \int_{1}^{\infty}\PP \bigg(\|  \bar{S}_N^1-\bar{B}^1_N \|_{\infty} > u N^{\rho}, \; A_N(u,\lambda)^c \bigg) \; d u.
\end{split}
\end{equation*}
For this purpose, we note the following result of P. Levy on the modulus of continuity of a Brownian motion.  (See the proof of  \cite[Theorem~ 8.4.2]{Durrett}.)

\begin{lemma}\label{lem:levy}
There are positive constants $K_1, K_2$ such that
\begin{equation}\label{scaledlevy}
	\PP \bigg( \sup_{\substack{s,t \in [0,1] \\ |t-s|\le r}} |B^1(s)-B^{1}(t)| > x  \bigg)
	\leq \frac{K_1}{r} e^{-K_2 \frac{x^2}{r}}
\end{equation}
for all $r\in (0,1)$ and $x>0$.
\end{lemma}

By \eqref{def:interpolation} and the triangle inequality,
\begin{equation}\label{eq:eq15}
\begin{split}
	&\PP \bigg( \|\bar S^1_N(t) - \bar B^1_N(t)\|_{\infty} > u N^{\rho}, \; A(u,\lambda)^c \bigg) \\
	& \leq   \PP \bigg( \sup_{t\in[0,1]} \left|B^1(\tau_1^1 + \cdots + \tau_{\lfloor Nt \rfloor}^1) - B^1(Nt)\right| > \frac12 uN^{\rho},
\;  A(u,\lambda)^c \bigg) \\
	&+ \PP \bigg( \sup_{t\in[0,1]} \left|B^1(\tau_1^1 + \cdots + \tau_{\lceil Nt \rceil}^1) - B^1(\tau_1^1 + \cdots + \tau_{\lfloor Nt \rfloor}^1) \right| > \frac12 uN^{\rho}, \;  A(u,\lambda)^c \bigg).
\end{split}
\end{equation}
On $A(\lambda,u)^c$ we have $|\tau_1^1 + \cdots + \tau_{\lfloor Nt \rfloor}^1 - \lfloor Nt \rfloor |\le  uN^{1-\lambda}$ for all $t\in [0,1]$.
Since $\lambda\in (0,1)$, we have $uN^{1-\lambda}\ge N^{1-\lambda}\ge 1$ for all $u\ge 1$ and $N\ge 1$.
Hence on $A(u,\lambda)^c$ we have for all $t \in [0,1]$ and  $N\geq 1$
\begin{equation*}
	|\tau_1^1 + \ldots + \tau_{\lfloor Nt \rfloor}^1 -Nt|\le  2uN^{1-\lambda}.
\end{equation*}

Thus, the first term in the right side of \eqref{eq:eq15} is bounded above by
\begin{equation}\label{eq:eq16}
\PP \bigg( \sup_{\substack{s,t \in [0,N+uN^{1-\lambda}] \\  |t-s|\le 2uN^{1-\lambda} }} |B^1(t) - B^1(s)| > \frac12 uN^{\rho} \bigg). 
\end{equation}

Similarly, the second term in the right side of \eqref{eq:eq15} is bounded above by
\begin{equation}\label{eq:eq16-1}
\PP \bigg( \sup_{\substack{s,t \in [0,N+uN^{1-\lambda}] \\  |t-s|\le 3uN^{1-\lambda} }} |B^1(t) - B^1(s)| > \frac12 uN^{\rho} \bigg).
\end{equation}
We must consider $s,t\in [0,N+uN^{1-\lambda}]$ because on the event $A(u,\lambda)^c$, $(\tau_1^1 + \ldots + \tau_{\lceil Nt \rceil})\leq~N+uN^{1-\lambda}$.

We wish to estimate the above probabilities using Lemma~\ref{lem:levy}. To do that we must rescale time by the factor $N+uN^{1-\lambda}$. Thus define $\tilde{t} = (N+uN^{1-\lambda})^{-1} t$. Then $B^{1}(t) = (N+uN^{1-\lambda})^{1/2} \tilde{B}^{1}(\tilde{t})$ where $\tilde{B}^{1}$ is another standard Brownian motion. This rescaling reduces \eqref{eq:eq16} to the left-hand side of \eqref{scaledlevy} with
$$ r= \frac{2 u N^{1-\lambda}}{N+uN^{1-\lambda}}, \qquad x=\frac{uN^{\rho}}{2 \sqrt{N+uN^{1-\lambda}}}.$$
Applying Lemma~\ref{lem:levy} shows that
$$\eqref{eq:eq16} \leq \frac{K_1}{2}\left(\frac{N^{\lambda}}{u} + 1\right) e^{-\frac{K_2}{8} u N^{2\rho - 1 +\lambda}}.$$
Likewise one sees that
$$\eqref{eq:eq16-1} \leq \frac{K_1}{3}\left(\frac{N^{\lambda}}{u} + 1\right) e^{-\frac{K_2}{12} u N^{2\rho - 1 +\lambda}}.$$

Integrating the sum of these bounds over the interval $u\in (1, \infty)$ and multiplying $N^{\rho}$, we find that for some other positive constants $K_3$ and $K_4$,
\begin{equation*}\label{eq:eq17}
	N^\rho \int_{1}^{\infty}\PP \bigg(\|  \bar{S}_N^1-\bar{B}^1_N \|_{\infty} > u N^{\rho}, \; A_N(u,\lambda)^c \bigg) \; d u
	\le K_3 N^{1-\rho} e^{-K_4 N^{2\rho-1+\lambda}}
\end{equation*}
for all $N\ge 1$.

Combining with \eqref{eq:eq1} and \eqref{eq:eq2},
we conclude that
\begin{equation}\label{eq:eq18}
 	\E \| \bar S_N^1-\bar{B}^1_N\|_{\infty} \leq
	N^\rho+  \text{Var} \; (\tau^1_1) N^{\rho- 1+2\lambda}
	+ K_3 N^{1-\rho} e^{-K_4 N^{2\rho-1+\lambda}}.
\end{equation}
Recall that $\text{Var} \; (\tau^1_1)<\infty$ by Lemma \ref{Skorohod}. Now choosing $\lambda= \frac12$ and $\rho=\frac14+\delta$ with $\delta>0$,  Lemma \ref{firstmainlemma} follows from  \eqref{eq:eq9} and \eqref{eq:eq18}.
\end{proof}

\begin{rem}\label{rem:KMTproof}
If we assume finite exponential moments for the weight distributions then we can better approximate the simple random walk by a Brownian motion using the dyadic approximation of Koml\'{o}s-Major-Tusn\'{a}dy. Indeed, using \cite[Theorem 7.1.1]{greg}, one sees that in this case for any $\delta >0$ there exists a constant $C>0$ such that
$$\E \| \bar S_N^1-\bar{B}^1_N\|_{\infty} \leq C N^\delta.$$
Then from (\ref{eq:eq9}) the term $N^{\frac{1}{4} + \delta}$ in Lemma \ref{firstmainlemma} is improved to $N^{\delta}$. This in turn implies the same change in the first term of the bound in Proposition \ref{compareprop}. It is easy to check that with this improved bound the proof of Theorem \ref{mainthm} goes through under the weaker assumption that $\alpha<\frac{3}{7}$.
\end{rem}

We now compare the partition function of the semi-discrete polymer with $F_N(\bar B_N )$.


\begin{lemma}\label{lem:estimate}
There are positive constants $K_1$ and $K_2$ such that
\begin{equation*}
	\PP \bigg( \left|F_N(\bar B_N) - \log \ZOY{n}{N}{\beta}\right| >a \bigg) \leq K_1nN e^{-\frac{K_2(a-\log (n!))^2}{4\beta^2n^2}}.
\end{equation*}
for all $n,N\in \N$ and $a>0$.
\end{lemma}

\begin{proof} The proof follows the same strategy of Lemma \ref{firstmainlemma}. We will use Lemma \ref{lem:levy}.
For $u >0$,  we define the event
$$ S(u) = \bigg \{  \sup_{\substack{s,t \in [0,N] \\ 0\leq |t'-s'|<1 \\ 1 \leq j \leq n }}
	|B^j(t) - B^j(s)| > u \bigg\}.$$
By Lemma \ref{lem:levy}, Brownian scaling, and the union bound, 
\begin{equation}\label{eq:bogus}
	\PP(S(u)) \leq K_1nN e^{-K_2 u^2}.
\end{equation}

%

Let us now observe that we can write
\begin{equation}\label{ineq:32}
\begin{split}
 	\ZOY{n}{N}{\beta} &=  \sum_{\substack{1=t_0\leq  t_1 \leq \cdots \leq t_{n-1}\leq t_n =N\\ t_i \in \Z}}
A(t_0, \ldots, t_n)
\end{split}
\end{equation}
where
$$ A(t_0, \ldots, t_n) :=\int_{D(t_0,\ldots, t_n)} e^{\beta E(\phi)} \; d\phi.$$
In order to define the domain $D(t_0,\ldots, t_n)$ over which $\phi$ is integrated, recall that $\phi$ can be specified by an ordered set $0=s_0<s_1<\cdots< s_{n-1}<s_n=N$. Then $D(t_0,\ldots, t_n)$ is the set of $\phi$ whose associated ordered set also satisfies $s_{i}\in [t_i -1, t_i)$ for $1\leq i \leq n-1$.

On the event $S(u)^c$, we have
\begin{equation*}
e^{ \beta \sum_{j=1}^n (B^j(t_{j}) - B^{j}(t_{j-1}-1))-  2\beta nu }
\leq e^{\beta E(\phi)} \leq  e^{ \beta \sum_{j=1}^n (B^j(t_{j}) - B^{j}(t_{j-1}-1))+ 2\beta nu}.
\end{equation*}

Integrating the above inequality over $\phi\in D(t_0,\ldots, t_n)$ and noting that the Lebesgue measure of $D(t_0,\ldots, t_n)$ is bounded in $[(n!)^{-1},1]$, we find that on $S(u)^c$,

\begin{equation*}\label{eq:eq20}
(n!)^{-1} e^{ \beta \sum_{j=1}^n (B^j(t_{j}) - B^{j}(t_{j-1}-1))-  2\beta nu }
\leq A(t_1, \ldots, t_n) \leq e^{ \beta \sum_{j=1}^n (B^j(t_{j}) - B^{j}(t_{j-1}-1))+ 2\beta nu}.
\end{equation*}

Hence on $S(u)^c$,
\begin{equation*}
 	F_N(\bar B_N) - 2\beta nu -\log (n!)\leq \log \ZOY{n}{N}{\beta} \leq F_N(\bar B_N)+ 2\beta nu .
\end{equation*}
Thus from~\eqref{eq:bogus} we find that
\begin{equation*}
	\PP \bigg( \left|F_N(\bar B_N) - \log \ZOY{n}{N}{\beta} \right| >2\beta n u +\log(n!) \bigg) \leq
	\PP(S(u)) \leq K_1nN e^{-K_2 u^2}.
\end{equation*}
for all $u>0$ and $n, N\in\N$.
Setting $a=2\beta n u +\log(n!)$, we obtain the claimed result of this lemma.
\end{proof}

The proof of Proposition \ref{compareprop} follows immediate by combining the above Lemmas \eqref{firstmainlemma} and \eqref{lem:estimate}.

\subsection{Proof of Proposition \ref{asyprop}}

This proof is based off of the proof of Theorem 4.1.46 of \cite{BorCor} which establishes the GUE limiting distribution for the semi-discrete polymer with a constant ratio between $n$ and $t$. Both that result and the present result rely upon Theorem 5.2.10 in \cite{BorCor} which (specializing that result to the case of $a_1=\cdots=a_N=0$) yields:


\begin{theorem}\label{OConFredDet}
Fix $n\geq 1$ and $\tau \geq 0$, and $0<\delta<1$. Then,
\begin{equation*}
\E \left[e^{-u \ZOY{n}{\tau}{1}}\right] = \det(I+ K_{u})
\end{equation*}
where $K_u: L^2(C_0)\to L^2(C_0)$ is the operator defined by the kernel
\begin{equation*}\label{Kudef}
K_{u}(v,v') = \frac{1}{2\pi \iota}\int_{-\iota \infty + \delta}^{\iota \infty +\delta}ds \Gamma(-s)\Gamma(1+s) \frac{\Gamma(v)^n}{\Gamma(s+v)^n} \frac{ u^s e^{v\tau s+ \tau s^2/2}}{v+s-v'}
\end{equation*}
and $C_{0}$ is a positively oriented contour containing $0$ and such that $|v-v'|<\delta$  for all $v,v'\in C_{0}$.
\end{theorem}

The above result for the Laplace transform leads to the desired asymptotic probability distribution by virtue of the following result which is given as Lemma 4.1.38 of \cite{BorCor}.

\begin{lemma}\label{problemma1}
Consider a one-parameter family of functions $\{f_t\}_{t\geq 0}$ mapping $\R\to [0,1]$ such that for each $t$, $f_t(x)$ is strictly decreasing in $x$ with a limit of $1$ at $x=-\infty$ and $0$ at $x=\infty$, and for each $\delta>0$, on $\R\setminus[-\delta,\delta]$ $f_t$ converges uniformly to ${\bf 1}(x\leq 0)$. Define the $r$-shift of $f_t$ as $f^r_t(x) = f_t(x-r)$.
Suppose that a one-parameter family of of random variables $X_t$ satisfies, for each $r\in \R$,
\begin{equation*}
\lim_{t\to \infty} \E[f^r_t(X_t)] = p(r)
\end{equation*}
for a continuous cumulative distribution function $p(r)$. Then $X_t$ converges weakly in distribution to a random variable $X$ which is distributed according to $\PP(X\leq r) = p(r)$.
\end{lemma}

Let $\mu>0$ and $\kappa>0$ be defined as in the hypothesis of Proposition \ref{asyprop}.
Consider the functions $f_t(x) = e^{-e^{t^{\mu}x}}$.
Observe that this family of functions meets the criteria of Lemma \ref{problemma1}.
By Lemma \ref{problemma1}, if for each $r\in \R$ we can prove that
\begin{equation*}
\lim_{t\to \infty} \E\left[f_t^r\bigg(\frac{\log \ZOY{\lfloor t^{\alpha}\rfloor}{\beta^2 t}{1} - 2\beta t^{1-\kappa}}{t^{\mu}}\bigg) \right] = F_{\rm{GUE}}(\beta^{-1} r),
\end{equation*}
then it will follow that
\begin{equation*}
\lim_{t\to \infty} \PP\bigg(\frac{\log \ZOY{\lfloor t^{\alpha}\rfloor}{\beta^2 t}{1} - 2\beta t^{1-\kappa}}{t^{\mu}}\le r \bigg) = F_{\rm{GUE}}(\beta^{-1} r).
\end{equation*}

Observe that
if we define
\begin{equation}\label{udef}
u=e^{-2\beta t^{1-\kappa} - rt^{\mu}},
\end{equation}
then
\begin{equation*}
	f_t^r\bigg(\frac{\log \ZOY{\lfloor t^{\alpha}\rfloor}{\beta^2 t}{1} - 2\beta t^{1-\kappa}}{t^{\mu}}\bigg)
	=e^{-u\ZOY{t^{\alpha}}{\beta^2 t}{1}}.
\end{equation*}
Hence, in view of Theorem \ref{OConFredDet}, our proof reduces now to proving that for $u$ as in (\ref{udef}), $n=\lfloor t^{\alpha}\rfloor$ and $\tau = \beta^2 t$
\begin{equation}\label{prlimit}
	\lim_{t\to \infty} \det(I+ K_{u})  = F_{\rm{GUE}}(\beta^{-1} r)
\end{equation}
for each fixed $r\in\R$.

We prove~\eqref{prlimit} using a steepest descent analysis. This analysis follows exactly the approach employed in the proof of Theorem 4.1.46 in \cite{BorCor}. As such we only include the critical point analysis here presently. The necessary contour manipulations and tail bounds can readily be found in \cite{BorCor} and easily adapted.
In the below we replace $n=\lfloor t^{\alpha}\rfloor$ by $n=t^{\alpha}$ to make the notations simpler.
It is straightforward to check that the error coming from this change is negligible

Let us consider the kernel $K_u(v, v')$.
We scale the kernel by setting $v=t^{-\kappa}\tilde{v}$ and $v'=t^{-\kappa}\tilde v'$
and define $\tilde{K}_u(\tilde v,\tilde v') = K_u(v, v') t^{-\kappa}$. Then
$\det(I+ K_{u})= \det(I+ \tilde{K}_{u})$.
By changing the variables as $s=t^{-\kappa}(\tilde{\zeta}-\tilde{v})$ and $\tilde{v}=t^{\kappa}v$ in the formula of the kernel,
we have
\begin{equation}\label{zetaKeqn}
	\tilde{K}_{u}(\tilde v,\tilde v')  = \frac{1}{2\pi \iota}\int \frac{\pi t^{-\kappa}}{\sin (\pi t^{-\kappa} (\tilde v-\tilde \zeta))}\exp\left\{t^{\alpha} \left(G(\tilde v)-G(\tilde \zeta)\right)+ r t^{\frac{\alpha}{3}}(\tilde v-\tilde \zeta)\right\}\frac{d\tilde \zeta}{\tilde \zeta-\tilde v'}
\end{equation}
where
\begin{equation*}
G( z) = \log \Gamma(t^{-\kappa} z) - \beta^2 z^2/2 + 2\beta  z
\end{equation*}
and we have also replaced $\Gamma(-s)\Gamma(1+s) = \pi / \sin(-\pi s)$.
Here the  contour for $\tilde\zeta$ in the integral may be taken as $\tilde v + \tilde \delta +\iota \R$
for any fixed $\tilde\delta>0$
and the operator $\tilde{K}_u$ may be defined on $L^2(\tilde{C}_0)$
where $\tilde{C}_0$  is a positively oriented contour containing $0$ and such that $|\tilde v-\tilde v'|<\tilde \delta$  for all $\tilde v,\tilde v'\in \tilde{C}_{0}$.
The problem is now prime for steepest descent analysis of the integral defining the kernel above.

The idea of steepest descent is to find critical points for the argument of the function in the exponential, and then to deform contours so as to go close to the critical point. The contours should be engineered such that away from the critical point, the real part of the function in the exponential decays and hence as $t$ gets large, has negligible contribution. This then justifies localizing and rescaling the integration around the critical point. The order of the first non-zero derivative (here third order) determines the rescaling in $t$ which in turn corresponds with the scale of the fluctuations in the problem we are solving. It is exactly this third order nature that accounts for the emergence of Airy functions and hence the Tracy Widom (GUE) distribution.

The Digamma function is defined as $\Psi(z)=[\log \Gamma]'(z)$ and, along with its first two derivatives, have the following expansion for $|z|$ small:
\begin{equation*}
\Psi(z) = \frac{-1}{z} + O(1),\qquad \Psi'(z) = \frac{1}{z^2} + O(1), \qquad \Psi''(z) = -\frac{2}{z^3} + O(1).
\end{equation*}
Let us then record the first three derivatives of $G$ along with their large $t$ expansions (afforded by the above expansions)
\begin{eqnarray*}
G'(\tilde v) =&  \Psi(t^{-\kappa} \tilde v)t^{-\kappa} - \beta^2 \tilde v + 2\beta& = -\frac{1}{\tilde v} - \beta^2 \tilde v +2\beta + O(t^{-\kappa})\\
G''(\tilde v) = & \Psi'(t^{-\kappa} \tilde v)t^{-2\kappa}-\beta^2 &= \frac{1}{\tilde v^2} - \beta^2 +  O(t^{-2\kappa})\\
G'''(\tilde v) = & \Psi''(t^{-\kappa} \tilde v)t^{-3\kappa} &= -\frac{2}{\tilde v^3} +  O(t^{-3\kappa}).
\end{eqnarray*}

The critical point is the value of $\tilde v=\tilde v_c$ at which $G'(\tilde v_c) =0$. However, it suffices to choose $\tilde v_c$ such that
\begin{equation*}
 -\frac{1}{\tilde v_c} - \beta^2 \tilde v_c +2\beta =0
\end{equation*}
since then $G'(\tilde v_c) = O(t^{-\kappa})$. Solving the above gives $\tilde v_c = \beta^{-1}$.
This choice of $\tilde v_c$ implies that 
\begin{eqnarray*}
G'(\tilde v_c) &=& O(t^{-\kappa})\\
G''(\tilde v_c) &= &  O(t^{-2\kappa})\\
G'''(\tilde v_c) &= &  -2\beta^3 + O(t^{-3\kappa}).
\end{eqnarray*}
Note that by taking $\tilde\delta$ small enough, it is possible to deform the contour of~\eqref{zetaKeqn}
to pass the critical point $\tilde \zeta= \beta^{-1}$.

Therefore we may make the final change of variables to expand around $\tilde v_c$ by setting $\tilde v = \tilde v_c + t^{-\frac{\alpha}{3}} \hat{v}$. From Taylor approximation and the above bounds on the derivatives of $G$ we find that
\begin{equation*}
G(\tilde{v}) = G(\tilde v_c) -\frac{\beta^3}{3} t^{-\alpha} (\hat{v})^3 + lot
\end{equation*}
where $lot$ denotes lower order terms in $t$.
We also set $\tilde \zeta= \tilde v_c + t^{-\frac{\alpha}{3}} \hat{\zeta}$.
We may conclude from the above argument that (note that the constants $G(\tilde v_c)$ cancel)
\begin{equation*}\label{Gexp}
t^{\alpha} \left(G(\tilde v)-G(\tilde \zeta)\right) = -\frac{\beta^3}{3}(\hat{v})^3 + \frac{\beta^3}{3}(\hat{\zeta})^3 + o(1).
\end{equation*}
Now we can turn to the remaining parts of the kernel. Note that
\begin{equation*}
r t^{\frac{\alpha}{3}}(\tilde v-\tilde \zeta) = r (\hat{v}-\hat{\zeta}).
\end{equation*}
As for the other two terms in $K_u$ we have
\begin{equation*}
 t^{-\frac{\alpha}{3}} \frac{\pi t^{-\kappa}}{\sin (\pi t^{-\kappa} (\tilde v-\tilde \zeta))}  \to  \frac{1}{\hat{v}-\hat{\zeta}}, \qquad \frac{d\tilde \zeta}{\tilde \zeta-\tilde v'} =\frac{d\hat{\zeta}}{\hat{\zeta}-\hat{v'}}
\end{equation*}
where the $ t^{-\frac{\alpha}{3}}$ comes from the Jacobian associated with the change of variables from $\tilde{v}$ to $\hat{v}$.

Putting together the pieces we have that (modulo the necessary tail estimates and uniformity which can be
obtained as in the proof of Theorem 4.1.46 in \cite{BorCor})
\begin{equation*}
\lim_{t\to \infty} \det(I+K_u) = \det(I+\hat{K}_r)
\end{equation*}
where
\begin{equation*}
\hat{K}_{r}: L^2(C_v)\to L^2(C_v)
\end{equation*}
for $C_{v}$ a contour given by rays from the origin at angles $\pm 2\pi/3$ oriented to have increasing imaginary part. The operator $\hat{K}_r$ is defined in terms of its integral kernel
\begin{equation*}\label{hatKrdef}
\hat{K}_{r}(\hat{v},\hat{v'}) = \frac{1}{2\pi \iota} \int_{C_{\zeta}} \frac{e^{ - \frac{\beta^3}{3} \hat{v}^3 +r\hat{v}}}{e^{ - \frac{\beta^3}{3} \hat{\zeta}^3 +r\hat{\zeta}}} \frac{d\hat{\zeta}}{(\hat{v}-\hat{\zeta})(\hat{\zeta}-\hat{v'})},
\end{equation*}
where $C_{\zeta}$ is a contour given by rays from $d$ (for some $d>0$) at angles $\pm \pi/3$ oriented to have increasing imaginary part.
But it is known that (see Proof of Theorem 4.1.46 in \cite{BorCor})
\begin{equation*}
\det(I+\hat{K}_r) =  F_{\rm{GUE}}(\beta^{-1} r).
\end{equation*}
This proves~\eqref{prlimit} and hence completes the proof of the proposition.


\end{document}